\documentclass{amsart}

\usepackage[dvips]{graphicx}

\usepackage[T1]{fontenc}
\usepackage[a4paper,top=46mm,bottom=46mm,left=37mm,right=37mm]{geometry}

\usepackage{latexsym,amsfonts} 
\usepackage{amsmath,amssymb,graphics,setspace}
\usepackage{amsthm}
\usepackage{mathrsfs} 
\usepackage[T1]{fontenc}%

\usepackage[ruled,vlined,linesnumbered]{algorithm2e}

\usepackage{marvosym}

\usepackage{pstricks,pst-text,pst-grad,pst-node,pst-3dplot,pstricks-add,pst-poly,pst-coil,pst-bar,pst-all,pst-plot,pst-2dplot} 
\usepackage{pst-fun,pst-blur} 

\usepackage{filecontents}

\usepackage{picins,graphicx}
\usepackage{paralist}
\usepackage{color}

\usepackage[verbose]{wrapfig}

\usepackage{subfigure}
\renewcommand{\thesubfigure}{\thefigure.\arabic{subfigure}}
\makeatletter
\renewcommand{\p@subfigure}{}
\renewcommand{\@thesubfigure}{\thesubfigure:\hskip\subfiglabelskip}
\makeatother

\usepackage{hyperref}
\hypersetup{linktocpage=true,colorlinks=true,linkcolor=blue,citecolor=blue,pdfstartview={XYZ 1000 1000 1}}

\usepackage{floatflt,graphicx}

\usepackage{graphicx}
\makeatletter
\DeclareFontFamily{U}{tipa}{}
\DeclareFontShape{U}{tipa}{bx}{n}{<->tipabx10}{}
\newcommand{\arc@char}{{\usefont{U}{tipa}{bx}{n}\symbol{62}}}%

\newcommand{\arc}[1]{\mathpalette\arc@arc{#1}}

\newcommand{\arc@arc}[2]{%
  \sbox0{$\m@th#1#2$}%
  \vbox{
    \hbox{\resizebox{\wd0}{\height}{\arc@char}}
    \nointerlineskip
    \box0
  }%
}
\makeatother

\makeatletter
\newcommand{\doublewedge}{\big@doubleop{\wedge}}
\newcommand{\big@doubleop}[1]{%
  \DOTSB\mathop{\mathpalette\big@doubleop@aux{#1}}\slimits@
}

\newcommand\big@doubleop@aux[2]{%
  \sbox\z@{$\m@th#1#2$}%
  \makebox[1.35\wd\z@][s]{$\m@th#1#2\hss#2$}%
}

\newcommand{\cl}{\mbox{cl}}

\newcommand{\Int}{\mbox{int}}
\newcommand{\bdy}{\mbox{bdy}}

\newcommand{\Nrv}{\mbox{Nrv}}
\newcommand{\dnear}{\delta_{\Phi}} % descriptive proximity
\newcommand{\dcap}{\mathop{\cap}\limits_{\Phi}} % descriptive intersection
\newcommand{\sh}{\mbox{sh}}
\newcommand{\cx}{\mbox{cx}}
\newcommand{\cyc}{\mbox{cyc}}
\newcommand{\vcyc}{\mbox{vcyc}}
\newcommand{\vtex}{\mbox{vtex}}
\newcommand{\vNrv}{\mbox{vNrv}}

\newcommand{\dfar}{{\not\delta}_{\Phi}} % descriptive far
\newcommand{\connCyc}{\mbox{conn}} % connectedness
\newtheorem{example}{Example}

\newtheorem{definition}{Definition}
\newtheorem{lemma}{Lemma}
\newtheorem{theorem}{Theorem}

\newtheorem{corollary}{Corollary}

\usepackage{xcolor}
\definecolor{light}{gray}{0.80}
\begin{document}
%\linenumbers

\title[Vortex Nerves]{Vortex Nerves and their Proximities.\\   Nerve Betti Numbers and Descriptive Proximity}

\author[James F. Peters]{James F. Peters}
\address{
Computational Intelligence Laboratory,
University of Manitoba, WPG, MB, R3T 5V6, Canada and
Department of Mathematics, Faculty of Arts and Sciences, Ad\.{i}yaman University, 02040 Ad\.{i}yaman, Turkey}
\thanks{The research has been supported by the Natural Sciences \&
Engineering Research Council of Canada (NSERC) discovery grant 185986 
and Instituto Nazionale di Alta Matematica (INdAM) Francesco Severi, Gruppo Nazionale per le Strutture Algebriche, Geometriche e Loro Applicazioni grant 9 920160 000362, n.prot U 2016/000036 and Scientific and Technological Research Council of Turkey (T\"{U}B\.{I}TAK) Scientific Human
Resources Development (BIDEB) under grant no: 2221-1059B211301223.}


\date{}

\dedicatory{Dedicated to  Enrico Betti and Som Naimpally}

\begin{abstract}
This article introduces vortex nerve complexes in CW (Closure finite Weak) topological spaces, which first appeared in works by P. Alexandroff, H. Hopf and J.H.C. Whitehead during the 1930s.    A \emph{vortex nerve} is a CW complex containing one or more intersecting path-connected cycles.   Each vortex nerve has its own distinctive shape.    Both vortex nerve shapes (bounded planar surfaces with nonempty interior) and holes (bounded planar surfaces with empty interior that live inside and define shapes) have boundaries that are path-connected cycles.   In the context of CW complexes, the usual Betti numbers $\mathcal{B}_0$ (cell count), $\mathcal{B}_1$ (cycle count) and $\mathcal{B}_2$ (hole count) provide a basis for the introduction of several new Betti numbers, namely, vortex $\mathcal{B}_{vtex}$, vortex nerve $\mathcal{B}_{vNrv}$ and shape $\mathcal{B}_{sh}$ introduced in this paper.   In addition, results are given for CW complexes equipped with a descriptive proximity as well as for the homotopy types of vortex nerves and the complexes and cycles contained in the nerves.   
\end{abstract}

\maketitle
\tableofcontents

\section{Introduction}
This paper introduces vortex nerve complexes in a CW topological space $K$.    A \emph{cell complex} is a nonempty collection of cells attached to each other in a Hausdorff space.   A \emph{cell} in the Euclidean plane is either a 0-cell (vertex) or 1-cell (edge) or 2-cell (filled triangle).   A nonvoid collection of cell complexes $K$ has a \emph{Closure finite Weak} (CW) topology, provided $K$ is Hausdorff (every pair of distinct cells is contained in disjoint neighbourhoods~\cite[\S 5.1, p. 94]{Naimpally2013}) and the collection of cell complexes in $K$ satisfy the Alexandroff-Hopf-Whitehead~\cite[\S III, starting on page 124]{AlexandroffHopf1935Topologie},~\cite[pp. 315-317]{Whitehead1939homotopy}, ~\cite[\S 5, p. 223]{Whitehead1949BAMS-CWtopology} containment and intersection conditions
\begin{description}
\item [\bf CW containment condition]  The closure of each cell complex $E$ (denoted by $\cl(\cx E)$) is a member of $K$.   The \emph{closure} $\cl(\cx E)$ equals the set of all cells on the boundary as well as in the interior of complex $\cx E$.
\item [\bf CW intersection condition]  The intersection of any two bounded cell complexes in $K$ is also in $K$.
\end{description}

Although not considered here, it also possible to work through the results for vortex nerves in a Leseberg bounded topological space~\cite{Leseberg2014QMimprovedNearness,Leseberg2013QMboundedProximities}, since every vortex nerve is a nonempty bounded set in a topological space of bounded sets~\cite{Leseberg2009FUboundedTopology}.   Each vortex cycle $A$ (denoted by $\vcyc A$ (briefly, vortex $\cyc A$)) is a collection of nesting, path-connected, filled cycles.    This paper continues the work on nerve complexes introduced in~\cite{Peters2018AMSBshapeComplexes}.

\begin{figure}[!ht]
\centering
\begin{pspicture}
%[showgrid=true]
(-1.5,-1.0)(8.0,2.5)
% cyc A = boudnary of a shape
\psline*[linestyle=solid,linecolor=gray!10]%
(2,2)(3,1.5)(3,0.5)(2,0)(1,0.5)(1,1.5)(2,2)
\psline[linestyle=solid](2,2)(3,1.5)\psline[linestyle=solid](3,1.5)(3,0.5)
\psline[linestyle=solid](3,0.5)(2,0)\psline[linestyle=solid](2,0)(1,0.5)
\psline[linestyle=solid](1,1.5)(2,2)\psline[arrows=<->](-1,1.7)(-0.3,2.0)\psline[arrows=<->](0.2,2.05)(0.8,1.75)
% cyc B = boundary of a hole
%\psline*[linestyle=solid,linecolor=gray!20]%
%(0,2)(1,1.5)(1,0.5)(0,0)(-1,0.5)(-1,1.5)(0,2)
\psline[linestyle=solid](0,2)(1,1.5)\psline[linestyle=solid](0,2)(-1,1.5)
\psline[linestyle=solid](-1,1.5)(-1,0.5)\psline[linestyle=solid](-1,0.5)(0,0)\psline[linestyle=solid](0,0)(1,0.5)
\psline[arrows=<->](1.2,1.75)(1.8,2.05)\psline[arrows=<->](2.2,2.05)(2.8,1.75)
\psdots[dotstyle=o, dotsize=2.6pt,linewidth=1.2pt,linecolor = black, fillcolor = yellow]%
(2,2)(3,1.5)(3,0.5)(2,0)(1,0.5)(1,1.5)(0,2)(1,1.5)(-1,1.5)(-1,0.5)(0,0)
\psline[linestyle=solid]{<->}(1,0.5)(1,1.5)
% inner vortex cycle cycB
\psline[linestyle=solid](6,2)(7,1.5)\psline[linestyle=solid,arrows=<->](5,1.5)(6,2)
\psline[linestyle=solid](5,1.5)(5,0.5)\psline[linestyle=solid](5,0.5)(6,0.0)
\psline[linestyle=solid](6,0.0)(7,0.5)\psline[linestyle=solid](7,0.5)(7,1.5)
% innermost vortex cycle cycA <a>
\psline*[linestyle=solid,linecolor=gray!10]%
(6,1.5)(6.5,1.3)(6.5,0.7)(6,0.5)(5.5,0.7)(5.5,1.3)(6,1.5)
\psline*[linestyle=solid,linecolor=white!90]%
(6,1.35)(6.35,1.23)(6.25,0.8)(6,0.8)(6,1.35)
%\psdots[dotstyle=o, dotsize=1.6pt,linewidth=1.2pt,linecolor = black, fillcolor = gray!80]%
%(6,1.35)(6.35,1.23)(6.25,0.8)(6,0.8)(6,1.35)
\psline[linestyle=solid, fillcolor = gray!30](6,1.35)(6.35,1.23)
\psline[linestyle=solid, fillcolor = gray!30](6.35,1.23)(6.25,0.8)
\psline[linestyle=solid, fillcolor = gray!30](6.25,0.8)(6,0.8)
\psline[linestyle=solid, fillcolor = gray!30](6,0.8)(6,1.35)
\psdots[dotstyle=o, dotsize=1.6pt,linewidth=1.2pt,linecolor = black, fillcolor = gray!80]%
(6,1.35)(6.35,1.23)(6.25,0.8)(6,0.8)(6,1.35)
\psline[linestyle=solid](6,1.5)(6.5,1.3)\psline[linestyle=solid,arrows=<->](6,1.5)(5.5,1.3)
\psline[linestyle=solid](6.5,1.3)(6.5,0.7)\psline[linestyle=solid](5.5,1.3)(5.5,0.7)
\psline[linestyle=solid](6.5,0.7)(6,0.5)\psline[linestyle=solid](5.5,0.7)(6,0.5)
% filament <e1>
\psline[arrows=<->,linestyle=solid,linecolor=red](5,1.5)(6.0,1.5)
% outermost vortex cycle cycC
\psline[linestyle=solid](6.0,2.8)(7.3,1.5)\psline[linestyle=solid,arrows=<->](4.7,1.5)(6.0,2.8)
\psline[linestyle=solid](4.7,1.5)(4.7,0.5)\psline[linestyle=solid](7.3,1.5)(7.3,0.5)
\psline[linestyle=solid](4.7,0.5)(6,-0.3)\psline[linestyle=solid](6,-0.3)(7.3,0.5)
% filament <e0> spanning cycC and cycB, intersecting cycA
\psline[arrows=<->,linestyle=solid,linecolor=red](6,1.5)(6.0,2.8)
\psdots[dotstyle=o, dotsize=2.6pt,linewidth=1.2pt,linecolor = black, fillcolor = yellow]%
(6,2)(7,1.5)(5,1.5)(5,0.5)(6,0.0)(7,0.5)(6,1.5)(6.5,1.3)(6.5,0.7)(5.5,0.7)(6,0.5)
(6.0,2.8)(7.3,1.5)(4.7,1.5)(4.7,0.5)(7.3,0.5)(4.7,0.5)(6,-0.3)(7.3,0.5)(5.5,1.3)
\rput(6,0.7){\footnotesize $\boldsymbol{cycA}$}\rput(5.8,1){\footnotesize $\boldsymbol{\langle a\rangle}$}
\rput(6.0,0.3){\footnotesize $\boldsymbol{cycB}$}\rput(6.5,0.4){\footnotesize $\boldsymbol{\langle b\rangle}$}
\rput(6.0,-0.5){\footnotesize $\boldsymbol{cycC}$}\rput(6.5,-0.2){\footnotesize $\boldsymbol{\langle c\rangle}$}
\rput(5.52,1.55){\footnotesize $\boldsymbol{\langle e_1\rangle}$}
\rput(6.27,2.26){\footnotesize $\boldsymbol{\langle e_0\rangle}$}
\rput(2.0,2.2){\footnotesize $\boldsymbol{0a}$}\rput(3.3,1.5){\footnotesize $\boldsymbol{1a}$}
\rput(3.3,0.5){\footnotesize $\boldsymbol{2a}$}
\rput(2,-0.2){\footnotesize $\boldsymbol{3a}$}\rput(1.13,0.7){\footnotesize $\boldsymbol{4b=4a}$}
\rput(1.13,1.4){\footnotesize $\boldsymbol{5b=5a}$}
\rput(0,-0.2){\footnotesize $\boldsymbol{3b}$}\rput(-1.3,0.5){\footnotesize $\boldsymbol{2b}$}
\rput(-1.3,1.5){\footnotesize $\boldsymbol{1b}$}
\rput(0.0,2.2){\footnotesize $\boldsymbol{0b}$}\rput(0.0,1){\footnotesize $\boldsymbol{cycB}$}
\rput(0.0,1.3){\footnotesize $\boldsymbol{\langle b\rangle}$}
\rput(2.0,1){\footnotesize $\boldsymbol{cycA}$}\rput(2.0,1.3){\footnotesize $\boldsymbol{\langle a\rangle}$}
\rput(1.0,-0.8){\footnotesize $\boldsymbol{(i)}$}\rput(6.0,-0.8){\footnotesize $\boldsymbol{(ii)}$}
\end{pspicture}
\caption[]{(i) Intersecting cycles $\cyc A, \cyc B$, $\cyc A\cap \cyc B$, (ii) cycles in a vortex nerve complex\\ 
$\vNrv E = \overbrace{\connCyc\left\{\mathop{\bigcup}\limits_{\cyc G\in \left\{\cyc A, \cyc B, \cyc C, \cyc e_0, \cyc e_1\right\}}\cyc G\right\}.}^{\mbox{\textcolor{blue}{\bf Path-connected cycles in a vortex nerve}}}$\\
%{\bf Symbol}: cyclic group generators: (i) $\boldsymbol{\langle a\rangle}$, $\boldsymbol{\langle b\rangle}$,
%(ii) $\boldsymbol{\langle a\rangle}$,$\boldsymbol{\langle e_1\rangle}$, $\boldsymbol{\langle e_0\rangle}$, $\boldsymbol{\langle b\rangle}$,
%$\boldsymbol{\langle c\rangle}$
}
\label{fig:vortexCycles}
\end{figure}

A \emph{planar vortex nerve} is a nonempty collection of intersecting filled cycles.   A \emph{planar filled cycle} $A$ (denoted by $\cyc A$) is a bounded region of the plane containing a non-void finite, collection $E$ of path-connected vertices so that there is a path between any pair of vertices in $\cyc A$.   There are two types of cycles, namely, shape boundary cycle and hole boundary cycles.   Each \emph{planar shape} is a finite, bounded region covered by a filled cycle.   A planar shape $\sh A$ has path-connected vertices on its boundary (denoted by $\bdy(\sh A)$) and a nonempty interior (denoted by $\Int(\sh A)$) that is either partially or completely filled.  A partially filled shape interior contains holes.  A \emph{planar hole} is a finite, bounded planar region with a cycle boundary and an empty interior.   A surface puncture in a piece of paper, the space bounded by a hoop and the space between the inner wall of a torus are examples of holes.   The interior of a hole contains no cells.   

A planar nerve in a CW complex is a collection of cells that have a common part.   Planar nerve complexes are examples of Edelsbrunner-Harer nerves.

\begin{definition}
Let $F$ be a finite collection of sets.   An \emph{\bf Edelsbrunner-Harer nerve}~\cite[\S III.2, p. 59]{Edelsbrunner1999} consists of all nonempty subcollections of $F$ (denoted by $\Nrv F$) whose sets have nonempty intersection, {\em i.e.},
\[
\Nrv F = \left\{X\subseteq F: \bigcap X\neq \emptyset\right\}. 
\]
\end{definition}

\begin{example}
Let cycles $\cyc A, \cyc B$ be represented in Fig.~\ref{fig:vortexCycles}(i) with a common part, namely, edge $\arc{4a,4b} = \arc{5a,5b}$.   Consequently,  $\cyc A \cap \cyc B\neq \emptyset$.   Hence, the collection $F = \left\{\cyc A, \cyc B\right\}$ is an Edelsbrunner-Harer nerve.    Notice that $F$ is also a collection a vortex nerves.
\qquad \textcolor{blue}{\Squaresteel}
\end{example}

\begin{definition}
Let $K$ be a finite collection of sub-complexes in a CW complex.   A \emph{\bf vortex nerve} consists of a nonempty collection $E$ of filled cycles $\cyc A$ in $K$ (denoted by $\vNrv E$) that have have nonempty intersection and which have zero or more edges attached between each pair of cycles in $\vNrv E$, {\em i.e.},
\[
\vNrv E = \left\{\cyc A\subseteq E: \bigcap \cyc A\neq \emptyset\right\}. \mbox{\qquad \textcolor{blue}{\Squaresteel}}
\]
\end{definition}

\begin{example}
A collection $\vNrv E$ of filled cycles = $
\left\{\cyc A, \cyc B, \cyc C, \cyc e_1, \cyc e_0\right\}$ is represented in Fig.~\ref{fig:vortexCycles}(ii).  
Each pair of cycles in this collection have nonempty intersection.   For instance,  $\cyc A \cap \cyc C\neq \emptyset$, since filled cycle $\cyc A$ is in the interior of the filled cycle $\cyc C$, {\em i.e.}, $\cyc A\in \Int(\cyc C)$.  Further, the intersection of all cycles in $\vNrv E$ is nonempty.  Hence, the collection $v\Nrv E$ of intersecting filled cycles is an Edelsbrunner-Harer nerve. 
\qquad \textcolor{blue}{\Squaresteel}
\end{example}

Holes define surface shapes.   A \emph{shape} is a finite, bounded planar region with a 1-cycle boundary and a nonempty interior.   By contrast, a \emph{hole} is a finite, bounded planar region with a 1-cycle boundary and an empty interior.  Planar holes are represented by bounded opaque regions.   A non-hole in the interior of a filled cycle is represented by a white planar region.   Typically, shapes such as bicycle tire (one hole) or a rabbit (many holes such as mouth, nostrils, ears) are defined by holes in their interiors.

\begin{example}
A massive planar hole is represented by the opaque region labelled $\cyc A$ in Fig.~\ref{fig:vortexCycles}(i).   A filled planar cycle containing a hole surrounding a non-hole region in its interior is represented by the opaque region labelled $\cyc A$ in Fig.~\ref{fig:vortexCycles}(ii).
\qquad \textcolor{blue}{\Squaresteel}
\end{example}

\section{Preliminaries}
This section briefly presents the notation, basic definitions and elementary lemmas and theorems for a descriptive proximity space.

%\begin{definition}
%A finite, bounded \emph{planar shape} $A$ (denoted by $\sh A$) is a finite region of the Euclidean plane bounded by a simple closed curve and with a nonempty interior~\cite{Peters2018AMSBshapeComplexes}.
%\end{definition}
%
%\begin{lemma}{\rm ~\cite{Peters2018AMSBshapeComplexes}}.\\
%Every vertex of a planar complex with three or more vertices is the nucleus of a nerve.
%\end{lemma}

%\subsection{Free Abelian Groups}
%\begin{lemma}{\rm ~\cite[\S 1.11, p. 53]{Munkres1984}}.\\
%Let $A$ be a free Abelian group of rank $n$.   If $B$ is a subgroup of $A$, then $B$ is free Abelain of rank $r\leq n$.
%\end{lemma} 

%\subsection{Descriptive Proximity Space}\label{sec:connectednessProximity}
A descriptive proximity space $X$ is defined in terms of a probe function $\Phi$ that maps each a nonempty subset $A$ in $X$ to a feature vector that describes $A$.   Let $2^X$ denote the collection of subsets in $X$.   
The mapping $\Phi:2^X\longrightarrow \mathbb{R}^n$ defined by
\[
\Phi(A) = \overbrace{\left(\mathbb{R}_1,\dots,\mathbb{R}_i,\dots,\mathbb{R}_n\right),}^{\mbox{\textcolor{blue}{\bf Feature vector that describes $A\in 2^X$ in Euclidean space $ \mathbb{R}^n$}}}
\]
where $\mathbb{R}_i$ is a real number.  For the axioms for a descriptive proximity, the usual set intersection $\cap$ for a traditional spatial proximity~\cite[\S 1, p. 7]{Naimpally70} is replaced by descriptive intersection $\dcap$~\cite[\S 3]{Peters2013mcs},~\cite{DiConcilio2018MCSdescriptiveProximities} (denoted by $\dcap$) defined by
\[
 A \dcap B = \{x \in A \cup B: \Phi(x) \in \Phi(A)\ \mbox{and}\ \ \Phi(x) \in \Phi(B) \}.
\]

If $A \dcap B$ is nonempty, there is at least one element of $A$ with a description that matches the description of an element of $B$.    It is entirely possible to identify a pair of nonempty sets $A, B\in 2^X$ separated spatially ({\em i.e.}, $A$ and $B$ have no members in common) and yet $A,B$ have matching descriptions.    The relation $\dnear$ reads \emph{descriptively close}.   We write $A\ \dnear\ B$, provided the feature vector that describes $A$ matches the feature vector that described $B$.
Let $X$ be equipped with the relation $\dnear$.   The pair $\left(X,\dnear\right)$ is a descriptive proximity space, provided the following axioms are satisfied.

\begin{description}
\item[{\rm\bf (dP0)}] $\emptyset\ \dfar\ A, \forall A \subset X$.
\item[{\rm\bf (dP1)}] $A\ \dnear\ B \Leftrightarrow B\ \dnear\ A$.
\item[{\rm\bf (dP2)}] $A\ \dcap\ B \neq \emptyset \Rightarrow\ A\ \dnear\ B$.
\item[{\rm\bf (dP3)}] $A\ \dnear\ (B \cup C) \Leftrightarrow A\ \dnear\ B $ or $A\ \dnear\ C$.
\end{description}

The converse of axiom {\bf dP2} also holds.

\begin{lemma}\label{lemma:dP2converse}
Let $X$ be equipped with the relation $\dnear$, $A,B\subset X$.   Then $A\ \dnear\ B$ implies $A\ \dcap\ B\neq \emptyset$.
\end{lemma}
\begin{proof}
$A\ \dnear\ B$ implies that there is a subset $E\subset A \cup B$ such that $\Phi(E) \in  \Phi(A)$ and $\Phi(E) \in  \Phi(B)$.   Hence, by definition, $A\ \dcap\ B\neq \emptyset$.    
\end{proof}

\begin{theorem}\label{theorem:dP2result}
Let $K$ be a collection of planar cell complexes equipped with the proximity $\dnear$, $\cx A,\cx B\in K$.   Then $\cx A\ \dnear\ \cx B$ implies $\cx A\ \dcap\ \cx B\neq \emptyset$.
\end{theorem}
\begin{proof}
Immediate from Lemma~\ref{lemma:dP2converse}. 
\end{proof}

\begin{corollary}\label{cor:dP2result}
Let $K$ be a collection of planar cell vortex nerves equipped with the proximity $\dnear$, $\vNrv A,\vNrv B\in K$.   Then $\vNrv\ \dnear\ \vNrv$ if and only if $\vNrv A\ \dcap\ \vNrv B\neq \emptyset$.
\end{corollary}
\begin{proof}$\mbox{}$\\
Each vortex nerve in $K$ is a planar cell complex.  Let $K$ be equipped with $\dnear$. \\
$\Rightarrow$:  From Theorem~\ref{theorem:dP2result}, $\vNrv A\ \dcap\ \vNrv B\neq \emptyset$.\\
$\Leftarrow$:  From Axiom dP2, $\vNrv A\ \dcap\ \vNrv B\neq \emptyset$ implies $\vNrv\ \dnear\ \vNrv$.
\end{proof}

\begin{lemma}\label{lemma:cellularProximitySpace}
Let $K$ be a collection of cell complexes equipped with the relation $\dnear$, cell complexes $\cx A,\cx B\subset K$.   Then $K$ is a descriptive proximity space.
\end{lemma}
\begin{proof}$\mbox{}$\\
Let $K$ be a collection cell complexes equipped with $\dnear$.   Then
(dP0): The empty set contains no cells.   Hence, $\emptyset\ \dfar\ \cx A, \forall\ \cx A \subset K$.\\
(dP1):  Assume $\cx A\ \dnear\ \cx B$.   Consequently, by definition of $\dnear$, $\cx A$ is descriptively close to $\cx B$, if and only if $\cx B$ is descriptively close to $\cx A$, if and only if $\cx B\ \dnear\ \cx A$.\\
(dP2):  Assume $\cx A\ \dcap\ \cx B \neq \emptyset$.   Consequently, there is at least one sub-complex of complex $\cx A$ with a description that matches the description of a sub-complex in complex $\cx B$, {\em i.e.}, $\cx A$ is descriptively close to $\cx B$.   Hence, $\cx A\ \dnear\ \cx B$.
(dP3): $\cx A\ \dnear\ (\cx B \cup \cx C)$ if and only if the feature vector $\Phi(\cx A)$ matches $\Phi(\cx B)$ or $\Phi(\cx C)$ (from the definition of $\dnear$ and $\cx B \cup \cx C$) if and only if $\cx A\ \dnear\ \cx B $ or $\cx A\ \dnear\ \cx C$.
Hence, $\left(K,\dnear\right)$ is a descriptive proximity space.
\end{proof}

\begin{theorem}\label{theorem:dP2vNRvResult}
Let $K$ be a collection of planar vortex nerves equipped with the proximity $\dnear$, $\cx A,\cx B\in K$.   Then $\left(K,\dnear\right)$ is a descriptive proximity space.
\end{theorem}
\begin{proof}
Immediate from Lemma~\ref{lemma:cellularProximitySpace}. 
\end{proof}

\begin{example}
Recall that $\mathcal{B}_2$ is a Betti number that counts the number of holes in a complex in a CW space $K$~\cite{Zomorodian2001BettiNumbers}.  Let $\Phi(\vNrv X) = (\mathcal{B}_2 )$ be a feature vector that describes a cell complex in a CW space $K$.      Also, let the pair of cell complexes that are vortex nerves in $K$ be represented by $\vNrv E = \left\{\cyc A,\cyc B\right\}$ in Fig.~\ref{fig:vortexCycles}(i) and $\vNrv E' = \left\{\cyc A,\cyc B,\cyc C\right\}$ in Fig.~\ref{fig:vortexCycles}(ii).   From Lemma~\ref{lemma:cellularProximitySpace} and Theorem~\ref{theorem:dP2vNRvResult}, axioms {\bf dP0, dP1, dP3} are satisfied for the space $K$.   Also observe that of the vortex nerves $\vNrv E, \vNrv E'$ contain one hole represented by an opaque sub-complex in each nerve, namely, $\cyc A$ in Fig.~\ref{fig:vortexCycles}(i) and $\cyc A$ in Fig.~\ref{fig:vortexCycles}(ii), {\em i.e.}, 
\[
\overbrace{\Phi(\vNrv E) = \mathcal{B}_2(\vNrv E) = \Phi(\vNrv E') =  \mathcal{B}_2(\vNrv E') = 1.}^{\mbox{\textcolor{blue}{\bf Both nerve complexes have the same number of holes}}}
\]
Consequently, $\vNrv E\ \dcap\ \vNrv E' \neq \emptyset$.   Then, $\vNrv E\ \dnear\ \vNrv E'$ (axiom {\bf dP2}) if and only $\vNrv E\ \dnear\ \vNrv E'$ implies $\vNrv E\ \dcap\ \vNrv E' \neq \emptyset$ (from Corollary~\ref{cor:dP2result}).    Hence, $\left(K,\dnear\right)$ is a descriptive proximity space.
\qquad \textcolor{blue}{\Squaresteel}
\end{example}

%\subsection{Free Abelian Groups}
%\begin{lemma}{\rm ~\cite[\S 1.11, p. 53]{Munkres1984}}.\\
%Let $A$ be a free Abelian group of rank $n$.   If $B$ is a subgroup of $A$, then $B$ is free Abelain of rank $r\leq n$.
%\end{lemma} 

\section{Main Results}
This section gives some main results for collections of proximal vortex cycles and proximal vortex nerves.

\subsection{Vortex, Vortex Nerve and Shape Betti numbers}
There are three basic types of Betti numbers, namely, $\mathcal{B}_0$ (number of cells in a complex),  $\mathcal{B}_1$ (number of cycles in a complex) and $\mathcal{B}_1$ (number of holes in a complex) ~\cite[\S 4.3.2, p. 57]{Zomorodian2001BettiNumbers}.   In terms of CW complexes, additional Betti numbers are useful, namely,$\mbox{}$\\
\begin{compactenum}[1$^o$]
\item [{\bf Vortex Betti number} $\boldsymbol{\mathcal{B}_{\mbox{\tiny \vtex}}}$ (number of vortex cycles) ] Let $\vtex E$ be a planar vortex, which is a collection of nesting, filled cycles.   For example, for a pair of nesting filled cycles $\cyc A \in \cyc B$, we have $\cyc A \in \Int(\cyc B)$.  Its Betti number is denoted by $\mathcal{B}_{\mbox{\tiny \vtex}}$.
\begin{example}
In Fig.~\ref{fig:vortexCycles}(ii), nesting cycles $\cyc A, \cyc \langle e\rangle, \cyc B, \cyc \langle e_0\rangle, \cyc C$ constitute a vortex and have $\mathcal{B}_{\mbox{\tiny \vtex}} = 5$.
\qquad \textcolor{blue}{\Squaresteel}
\end{example}
\item [{\bf Vortex nerve Betti number} $\boldsymbol{\mathcal{B}_{\mbox{\tiny \vNrv}}}$ (number of vortex nerve cells, cycles and holes)] Let $\vNrv E'$ be a planar vortex nerve.   Its Betti number is denoted by $\mathcal{B}_{\mbox{\tiny \vNrv}}$.$\mbox{}$\\
\vspace{2mm}
\item [{\bf Shape Betti number}  $\boldsymbol{\mathcal{B}_{\mbox{\tiny \sh}}}$ (number of shape cycles and holes)] Let $\sh E$ be a planar shape.   Its Betti number is denoted by $\mathcal{B}_{\mbox{\tiny \sh}}$.
\begin{example}
In Fig.~\ref{fig:vortexCycles}(ii), a shape $\sh E$ is defined by the vortex nerve $\vNrv E$, which contains 5 nesting cycles $\cyc A, \cyc \langle e_1\rangle, \cyc B, \cyc \langle e_0\rangle, \cyc C$.   Notice that cycles $\cyc \langle e_1\rangle, \cyc \langle e_0\rangle$ are bi-directional edges.   In addition, the interior of $\cyc A$ contains a hole (represented by the opaque sub-region in $\cyc A)$.  Hence, $\mathcal{B}_{\mbox{\tiny \sh}} = 2 + 5 + 1 = 8$.
\qquad \textcolor{blue}{\Squaresteel}
\end{example}
\end{compactenum}

\begin{lemma}\label{lemma:BettiNumbers}
Let $\mathcal{B}_{0},\mathcal{B}_{1},\mathcal{B}_{2}$ be Betti numbers that count the number of cells, number of cycles and number of holes  in a planar CW complex, respectively.   Then
\begin{compactenum}[1$^o$]
\item $\mathcal{B}_{\mbox{\tiny \vtex}} = \mathcal{B}_{1}$.
\item $\mathcal{B}_{\mbox{\tiny \vNrv}} = \mathcal{B}_{0} +\mathcal{B}_{1} + \mathcal{B}_{2}$.
\item $\mathcal{B}_{\mbox{\tiny \sh}} = \mathcal{B}_{1} + \mathcal{B}_{2}$ = 1 + $\mathcal{B}_{2}$ for a non-vortex nerve shape
\end{compactenum}
\end{lemma}
\begin{proof}$\mbox{}$\\
1$^o$: By definition, a \emph{vortex} is a collection of nesting, non-concentric cycles.   Consequently, $\mathcal{B}_{\mbox{\tiny \vtex}}$ is a count of the number cycles in the vortex, which is the same as the oneth Betti number $\mathcal{B}_{1}$. \\
2$^o$: By definition, a \emph{vortex nerve} $\vNrv E$ is a collection of nesting, non-concentric connected cycles with common part.   At least one edge is attached between each pair of neighbouring cycles.   The innermost cycle of  $\vNrv E$ is either the boundary of a hole or the boundary of a nonempty planar region containing a hole.   It is also possible for one or more nerve cycles to be boundaries of  planar regions containing holes in addition to the possible hole in the innermost nerve cycle.   That is, each cycle in $\vNrv E$ can be the boundary of a planar region containing contain 0 or more holes.  Consequently, $\mathcal{B}_{\mbox{\tiny \vNrv}}$ is a count of the number of edges attached between nerve cycles plus the number of nerve cycles plus the number of nerve holes.   Hence, $\mathcal{B}_{\mbox{\tiny \vNrv}} = \mathcal{B}_{0} +\mathcal{B}_{1} + \mathcal{B}_{2}$.\\
3$^o$:  A planar shape $\sh E$ is defined in terms of its boundary and nonempty interior containing zero or more punctures (holes).  Since each shape has a single cycle on its boundary, $\mathcal{B}_{1}(\sh E)$ = 1.   Hence, $\mathcal{B}_{\mbox{\tiny \sh}} = 1 + \mathcal{B}_{2}$.
\end{proof}

\subsection{Topology on Vortex Cycle Spaces}
This section introduces the construction of topology classes of vortex cycles and vortex nerves.  Topology classes have proved to be useful in classifying physical objects such as quasi-crystals~\cite{Dareau2017arXivCrystalTopClasses} and in knowledge extraction~\cite{Fermi2018knowledgeExtraction}. Such classes provide a basis for knowledge extraction about proximal vortex cycles and nerves.   A strong beneficial side-effect of the construction of such classes is the ease with which the persistence of homology class objects can be computed (see, {\em e.g.},~\cite{Fermi2017arXivPersistentTopology},~\cite{Taimanov2017LNAIoilGasPersistence}).   More importantly, the construction of topology classes leads to problem size reduction (see, {\em e.g.},~\cite[\S 3.1, p. 5]{Pellikka2010computationalHomologyElectromagnetics}).

\begin{lemma}\label{lemma:vortexNerveParts}
Let $K$ be a nonempty collection of cell complexes $K$ that is a Hausdorff space with a Closure finite Weak (CW) topology.   Then
\begin{compactenum}[1$^o$]
\item A single 0-cell (vertex) in $K$ is a vortex nerve.
\item A single 1-cell (bi-directional edge)  in $K$ is a vortex nerve.
\item A single 2-cell (filled triangle with bi-directional, path-connected edges)  in $K$ is a vortex nerve.
\item A single path-connected cycle in $K$ is a vortex nerve.
\item Let $\cx E$ be a pair of filled cycles $\cyc A, \cyc B$ in $K$ such that $\cyc A \subset \Int(\cyc B)$ and $\cyc A\cap \cyc B\neq \emptyset$.  Then $\cx E$ is a vortex nerve.
\end{compactenum}
\end{lemma}
\begin{proof}$\mbox{}$\\
1$^o$:  Let $\cx E\in K$ be a cell complex containing a single 0-cell $x$.   $\cl(\cx E) = \cl(x)\in K$ (containment property) and $\cx E\cap \cx E = \cx E\in K$ (intersection property).   By definition, $\cx E$ is a cycle, since there is a path between every pair of vertexes in $\cx E$, {\em i.e.}, path between 0-cell and itself.  Hence, $\cx E$ is a vortex nerve, since $\cx E\cap \cx E = \cx E\in K$. \\
2$^o$: Let $\cx E\in K$ be a cell complex containing a single 1-cell $\cx A = \arc{a,b}$.   $\cx A$ is a cycle, since there is a path from either vertex $a$ to $b$ or from $b$ to $a$, since $\cx A$ is bi-directional.   We also have $vNrv(\cx E) = \left\{\cx A\in \cx E: \bigcap \cx A = \cx A\neq \emptyset\right\}$.  Hence, edge $\cx A$ is a cycle and $vNrv(\cx E)$ is a vortex nerve.\\
3$^o$: Let $\cx E\in K$ be a cell complex containing a single 2-cell $\cx A = \bigtriangleup{A}$.  By definition, the boundary $\bdy(\bigtriangleup{A})$ is a cycle, since each pair of vertices on $\bigtriangleup{A}$ is path-connected.   Further, $vNrv(\cx E) = \left\{\bigtriangleup{A}\in \cx E: \bigcap \bigtriangleup{A} = \bigtriangleup{A}\neq \emptyset\right\}$, {\em i.e.}, $\bigtriangleup{A}$ intersects with itself.   Hence, $vNrv(\cx E)$ is a vortex nerve.\\
4$^o$: The proof that a path-connected cycle in $K$ is a vortex nerve is symmetric with the proof of part 3$^o$.\\
5$^o$:  Let $\cx E\in K$ be a cell complex containing a pair of nesting filled cycles $\cyc A, \cyc B$ in $K$ such that $\cyc A\in \Int(\cyc B)$.   We have
$vNrv(\cx E) = \left\{\cyc A\in \cx E: \bigcap \cyc A = \cyc B\neq \emptyset\right\}$, since $\cyc A$ is in the interior of $\cyc B$ and both cycles are filled.   That is, $\cyc A\subseteq \cyc B$.   Hence, $vNrv(\cx E)$ is a vortex nerve.
\end{proof}

\begin{theorem}\label{theorem:descriptiveCWtopology}
Let $K$ be a nonempty finite collection of planar vortex nerve complexes that is a Hausdorff space equipped the proximity $\dnear$.  From the pair $\left(K,\dnear\right)$, a Closure finite Weak (CW) topology can be constructed. 
\end{theorem}
\begin{proof}$\mbox{}$\\
From Theorem~\ref{theorem:dP2vNRvResult}, $\left(K,\dnear\right)$ is a descriptive proximity space.  Let $\vNrv A, \vNrv B$ be vortex nerve complexes in $K$.  The closure $\cl(\vNrv A)$ is finite and includes the boundary $\bdy(\vNrv A)$ is a path-connected cycle in $\vNrv A$.   From Lemma~\ref{lemma:vortexNerveParts}(4), $\bdy(\vNrv A)$ is a vortex nerve.  Consequently, $\cl(\vNrv A)$ is in $K$ (CW containment property).

Since $K$ is finite, $\cl(\vNrv A)$ intersects a only a finite number of other vortex nerves in $K$.   The descriptive intersection $\vNrv A \dcap \vNrv B\neq \emptyset$ is itself is either either a 0-cell (vertex) or a 1-cell (bi-directional edge)of a 2-cell (filled triangle with bi-directional edges) or a 1-cycle or a sub-complex containing connected vertexes.   From Lemma~\ref{lemma:vortexNerveParts}, the nonempty intersection $\vNrv A \dcap \vNrv B$ is itself a vortex nerve (CW intersection property).   Hence, $\left(K,\dnear\right)$ defines a Whitehead CW topology.
\end{proof}

\begin{theorem}\label{thm:vortexSconnCWtopology}
Let $K$ be a nonempty collection of finite cell complexes $K$ that is a Hausdorff space equipped the proximity $\dnear$.  From the pair $\left(K,\dnear\right)$, a Closure finite Weak (CW) topology can be constructed. 
\end{theorem}
\begin{proof}$\mbox{}$\\
Replace $\dcap$ with $\cap$ in the proof of Theorem~\ref{theorem:descriptiveCWtopology} and the desired result follows.
\end{proof}

\begin{theorem}
Let $K$ be a nonempty finite collection of planar vortex nerve complexes that is a Hausdorff space, $\vNrv E\in K$.   Then
\begin{compactenum}[1$^o$]
\item If $\vNrv E$ is a collection of $k > 0$ intersecting filled cycles with zero attached edges between the cycles and no holes in the cycle interiors, then $\mathcal{B}_{\mbox{\tiny \vNrv}}(\vNrv E) = \mathcal{B}_{\mbox{\tiny \vtex}} = \mathcal{B}_{1} = k$.
\item  If $\vNrv E$ is a collection of $k > 0$ intersecting filled cycles with zero attached edges between the cycles and $n$ holes in the cycle interiors, then $\mathcal{B}_{\mbox{\tiny \vNrv}}(\vNrv E) = \mathcal{B}_{\mbox{\tiny \sh}} = \mathcal{B}_{1}+  \mathcal{B}_{2} = k + n$.
\item If $\vNrv E$ contains a single filled cycle with zero holes in its interior, then $\mathcal{B}_{\mbox{\tiny \vNrv}}(\vNrv E) = \mathcal{B}_{\mbox{\tiny \vtex}} = \mathcal{B}_{1} = 1$.
\end{compactenum}
\end{theorem}
\begin{proof}$\mbox{}$\\
1$^o$: From Lemma~\ref{lemma:BettiNumbers}(2$^o$), 
\begin{align*}
\mathcal{B}_{\mbox{\tiny \vNrv}}(\vNrv E) &= \mathcal{B}_{0} +\mathcal{B}_{1} + \mathcal{B}_{2}\\
   &= 0 +\mathcal{B}_{1} + \mathcal{B}_{2},\ \mbox{since $\vNrv E$ has zero edges between its cycles}\\
	 &= 0 +\mathcal{B}_{1} + 0,\ \mbox{since $\vNrv E$ has zero holes in its interior}\\
	 &= 0 + k + 0 = k,\ \mbox{since $\vNrv E$ contains $k$ intersecting cycles}\\
\end{align*}
2$^o$, 3$^o$:  The proofs parts 2$^o$ and 3$^o$ are similar to the proof of part 1$^o$.
\end{proof}

%\begin{theorem}
%Let $K$ be a nonempty finite collection of planar vortex nerve complexes that is a Hausdorff space.   A Closure finite Weak (CW) topology can be constructed on $K$.
%\end{theorem}

\subsection{Homotopic Types of Vortex Cycles and Vortex Nerves}

\begin{theorem}\label{EHnerve}{\rm ~\cite[\S III.2, p. 59]{Edelsbrunner1999}}
Let $\mathscr{F}$ be a finite collection of closed, convex sets in Euclidean space.  Then the nerve of $\mathscr{F}$ and the union of the sets in $\mathscr{F}$ have the same homotopy type.
\end{theorem}

\begin{theorem}\label{thm:vortexNerveHomotopy}
Let $K$ be a collection of closed, convex cell complexes $\cx A$ in Euclidean space.
Then 
\begin{compactenum}[1$^o$]
\item Each nerve complex $\Nrv E = \left\{\cx A\in K: \bigcap \cx A \neq \emptyset\right\}$ and the union of the cell sub-complexes in $\cx A$ in $\Nrv E$ have the same homotopy type.
\item Each vortex nerve $\vNrv E = \left\{\cyc A\in K: \bigcap \cx A \neq \emptyset\right\}$ and the union of the vortex nerve sub-cycles $\cyc A\in \vNrv E$ have the same homotopy type.
\end{compactenum}
\end{theorem}
\begin{proof}$\mbox{}$\\
1$^o$: From Theorem~\ref{EHnerve}, we have that the union of the cell complexes $\cx A\in \vNrv E$ and $\vNrv E$ have the same homotopy type.\\
2$^o$: From Lemma~\ref{lemma:vortexNerveParts}, each cycle $\cyc A$ in a vortex nerve $\vNrv E$ is a vortex nerve.   Hence, from Theorem~\ref{EHnerve}, $\vNrv E$ and the union of cycles $\cyc A$ in $\vNrv E$ have the same homotopy type.
\end{proof}

\begin{theorem}\label{thm:dnearVortexNerveHomotopy}
Let $K$ be a collection of vortex nerves that are closed, convex complexes in Euclidean space equipped with $\dnear$.
Then the vortex nerve $\vNrv E = \left\{\cyc A\in K: \mathop{\bigcap}\limits_{\Phi} \cyc A \neq \emptyset\right\}$ and the union of the cycles $\cyc A$ in $\vNrv E$ have the same homotopy type.
\end{theorem}
\begin{proof}$\mbox{}$\\
Replace $\bigcap$ with $\mathop{\bigcap}\limits_{\Phi}$ in the proof of Theorem~\ref{thm:vortexNerveHomotopy}(2$^o$) and the desired result follows.
\end{proof}

\section*{Acknowledgements}
Many thanks to the anonymous reviewer and to Professor Tane Vergili for their helpful suggestions and corrections.

\bibliographystyle{amsplain}
\bibliography{NSrefs}

\end{document}